\documentclass[a4paper,12pt,final]{amsart}
\usepackage{times,a4wide,mathrsfs,amssymb,amsmath,amsthm}

\newcommand{\C}{\mathbb{C}}

\newcommand{\QQ}{\mathbb{Q}}
\newcommand{\NN}{\mathbb{N}}
\newcommand{\PP}{\mathbb{P}}

\newcommand{\OO}{\mathcal O}

\newcommand{\MM}{\mathcal M}

\newcommand{\wt}{\widetilde}

\DeclareMathOperator{\aut}{Aut}
\DeclareMathOperator{\ide}{id}

\newtheorem{theorem}{Theorem}[section]
\newtheorem{claim}[theorem]{Claim}
\newtheorem{lemma}[theorem]{Lemma}
\newtheorem{sublemma}[theorem]{Sublemma}
\newtheorem{corollary}[theorem]{Corollary}
\newtheorem{proposition}[theorem]{Proposition}
\newtheorem{conjecture}[theorem]{Conjecture}
\newtheorem{remark}[theorem]{Remark}
\newtheorem{definition}[theorem]{Definition}
\newtheorem{convention}{Conventions}

\newtheorem{nonumbering}{Theorem}

\newtheorem{nonumberingc}{Corollary}

\newtheorem{nonumberingt}{Acknowledgements}

\begin{document}
\author[Robert Laterveer]
{Robert Laterveer}

\address{Institut de Recherche Math\'ematique Avanc\'ee,
CNRS -- Universit\'e 
de Strasbourg,\
7 Rue Ren\'e Des\-car\-tes, 67084 Strasbourg CEDEX,
FRANCE.}
\email{robert.laterveer@math.unistra.fr}

\title[Algebraic cycles on certain HK fourfolds]{Algebraic cycles on certain hyperk\"ahler fourfolds with an order $3$ non--symplectic automorphism}

\begin{abstract} Let $X$ be a hyperk\"ahler variety, and assume $X$ has a non--symplectic automorphism $\sigma$ of order $>{1\over 2}\dim X$. Bloch's conjecture predicts that the quotient
$X/<\sigma>$ should have trivial Chow group of $0$--cycles. We verify this for Fano varieties of lines on certain special cubic fourfolds
having an order $3$ non--symplectic automorphism.
\end{abstract}


\maketitle

\section{Introduction}

Let $X$ be a smooth projective variety over $\C$, and let $A^i(X):=CH^i(X)_{\QQ}$ denote the Chow groups of $X$ (i.e. the groups of codimension $i$ algebraic cycles on $X$ with $\QQ$--coefficients, modulo rational equivalence). Let $A^i_{hom}(X)$ denote the subgroup of homologically trivial cycles. It does not seem an exaggeration to say that the field of algebraic cycles is filled with open questions \cite{B}, \cite{J2}, \cite{J4}, \cite{MNP}, \cite{Vo}. Among these open questions, a prominent position is occupied by Bloch's conjecture, proudly and sturdily overtowering the field like an unscalable mountain top.

\begin{conjecture}[Bloch \cite{B}]\label{bloch1} Let $X$ be a smooth projective variety of dimension $n$. Let $\Gamma\in A^n(X\times X)$ be such that
  \[ \Gamma_\ast=0\colon\ \ \ H^i(X,\OO_X)\ \to\ H^i(X,\OO_X)\ \ \ \forall i>0\ .\]
  Then
  \[ \Gamma_\ast=0\colon\ \ \ A^n_{hom}(X)\ \to\ A^n(X)\ .\]
  \end{conjecture}
  
A particular case of conjecture \ref{bloch1} is the following:
 
  \begin{conjecture}[Bloch \cite{B}]\label{bloch2} Let $X$ be a smooth projective variety of dimension $n$. Assume that
  \[  H^i(X,\OO_X)=0\ \ \ \forall i>0\ .\]
  Then
  \[  A^n_{}(X)\cong\QQ\ .\]
  \end{conjecture}
  
 The ``absolute version'' (conjecture \ref{bloch2}) is obtained from the ``relative version'' (conjecture \ref{bloch1}) by taking $\Gamma$ to be the diagonal. Conjecture \ref{bloch2} is famously open for surfaces of general type (cf. \cite{PW}, \cite{V8}, \cite{Gul} for some recent progress).

Let us now suppose that $X$ is a hyperk\"ahler variety (i.e., a projective irreducible holomorphic symplectic manifold \cite{Beau0}, \cite{Beau1}), say of dimension $2m$. Suppose there exists a non--symplectic automorphism $\sigma\in\aut(X)$ of order $k>m$. This implies that
  \[    \bigl( \sigma+\sigma^2 + \ldots +\sigma^k\bigr){}_\ast=0\colon\ \ \ H^i(X,\OO_X)\ \to\ H^{i}(X,\OO_X)\ \ \ \forall i>0\ .\]
  
 Conjecture \ref{bloch1} (applied to the correspondence $\Gamma=\sum_{j=1}^k \Gamma_{\sigma^j}\in A^{2m}(X\times X)$, where $\Gamma_f$ denotes the graph of an automorphism $f\in\aut(X)$) then predicts the following:

\begin{conjecture}\label{conjhk} Let $X$ be a hyperk\"ahler variety of dimension $2m$. Let $\sigma\in\aut(X)$ be an order $k$ non--symplectic automorphism, and assume $k>m$. Then
  \[  \bigl( \sigma+\sigma^2 + \ldots +\sigma^k\bigr){}_\ast=0\colon\ \ \ A^{2m}_{hom}(X)\ \to\ A^{2m}(X)\ .\]
 \end{conjecture}
  
 The main result of this note is that conjecture \ref{conjhk} is true for a certain family of hyperk\"ahler fourfolds:
  
  \begin{nonumbering}[=theorem \ref{main}] Let $Y\subset\PP^5(\C)$ be a smooth cubic fourfold defined by an equation
    \[ f(X_0,X_1,X_2,X_3)+ g(X_4,X_5)=0\ ,\]
    where $f$ and $g$ are homogeneous polynomials of degree $3$. Let $X=F(Y)$ be the Fano variety of lines in $Y$. 
    Let $\sigma\in\aut(X)$ be the order $3$ automorphism induced by
    \[  \begin{split} \PP^5(\C)\ &\to\ \PP^5(\C)\ ,\\
                         [X_0:\ldots:X_5]\ &\mapsto\ [X_0:X_1:X_2:X_3:\nu X_4:\nu X_5]\\
                         \end{split}\]
    (where $\nu$ is a primitive $3$rd root of unity). 
    
    Then
    \[ (\ide +\sigma+\sigma^2)_\ast \ A^4_{hom}(X)=0\ .\]
  \end{nonumbering} 
  
  As an immediate consequence of theorem \ref{main}, we find that Bloch's conjecture \ref{bloch2} is verified for the quotient:
  
  \begin{nonumberingc}[=corollary \ref{triv}] Let $X$ and $\sigma$ be as in theorem \ref{main}, and let $Z:=X/<\sigma>$ be the quotient. Then
    \[  A^4(Z)\cong\QQ\ .\]
    \end{nonumberingc}
    
    Another consequence (corollary \ref{ghc}) is that a certain instance of the generalized Hodge conjecture is verified.
   
       The proof of theorem \ref{main} relies on the theory of finite--dimensional motives \cite{Kim}, combined with the Fourier decomposition of the Chow ring of $X$ constructed by Shen--Vial \cite{SV}.

 \vskip0.6cm

\begin{convention} In this article, the word {\sl variety\/} will refer to a reduced irreducible scheme of finite type over $\C$. A {\sl subvariety\/} is a (possibly reducible) reduced subscheme which is equidimensional. 

{\bf All Chow groups will be with rational coefficients}: we will denote by $A_j(X)$ the Chow group of $j$--dimensional cycles on $X$ with $\QQ$--coefficients; for $X$ smooth of dimension $n$ the notations $A_j(X)$ and $A^{n-j}(X)$ are used interchangeably. 

The notations $A^j_{hom}(X)$, $A^j_{AJ}(X)$ will be used to indicate the subgroups of homologically trivial, resp. Abel--Jacobi trivial cycles.
For a morphism $f\colon X\to Y$, we will write $\Gamma_f\in A_\ast(X\times Y)$ for the graph of $f$.
The contravariant category of Chow motives (i.e., pure motives with respect to rational equivalence as in \cite{Sc}, \cite{MNP}) will be denoted $\MM_{\rm rat}$.



We will write $H^j(X)$ 
to indicate singular cohomology $H^j(X,\QQ)$.

\end{convention}

\section{Preliminaries}

\subsection{Quotient varieties}
\label{ssquot}

\begin{definition} A {\em projective quotient variety\/} is a variety
  \[ Z=X/G\ ,\]
  where $X$ is a smooth projective variety and $G\subset\hbox{Aut}(X)$ is a finite group.
  \end{definition}
  
 \begin{proposition}[Fulton \cite{F}]\label{quot} Let $Z$ be a projective quotient variety of dimension $n$. Let $A^\ast(Z)$ denote the operational Chow cohomology ring. The natural map
   \[ A^i(Z)\ \to\ A_{n-i}(Z) \]
   is an isomorphism for all $i$.
   \end{proposition}
   
   \begin{proof} This is \cite[Example 17.4.10]{F}.
      \end{proof}

\begin{remark} It follows from proposition \ref{quot} that the formalism of correspondences goes through unchanged for projective quotient varieties (this is also noted in \cite[Example 16.1.13]{F}). We may thus consider motives $(Z,p,0)\in\MM_{\rm rat}$, where $Z$ is a projective quotient variety and $p\in A^n(Z\times Z)$ is a projector. For a projective quotient variety $Z=X/G$, one readily proves (using Manin's identity principle) that there is an isomorphism of motives
  \[  h(Z)\cong h(X)^G:=(X,\Delta_G,0)\ \ \ \hbox{in}\ \MM_{\rm rat}\ ,\]
  where $\Delta_G$ denotes the idempotent ${1\over \vert G\vert}{\sum_{g\in G}}\Gamma_g$.  
  \end{remark}

\subsection{Finite--dimensional motives}

We refer to \cite[Definition 3.7]{Kim} for the definition of finite--dimensional motive (cf. also \cite{An}, \cite{Iv}, \cite{J4}, \cite[Chapters 4 and 5]{MNP} for further context and applications). 
The following two results provide a lot of examples:

\begin{theorem}[Kimura \cite{Kim}]\label{domcurves} Let $X$ be a smooth projective variety, and assume $X$ is dominated by a product of curves. Then $X$ has finite--dimensional motive.
\end{theorem}

\begin{proof} A smooth projective curve has finite--dimensional motive \cite[Corollary 4.4]{Kim}. Since finite--dimensionality is stable under taking products of varieties \cite[Corollary 5.11]{Kim}, a product of curves has finite--dimensional motive. Applying \cite[Proposition 6.9]{Kim}, this implies that $X$ has finite--dimensional motive.
\end{proof}

\begin{theorem}\label{blow} Let $X$ be a smooth projective variety, and let $\wt{X}$ be the blow--up of $X$ with smooth center $Y\subset X$. If $X$ and $Y$ have finite--dimensional motive, then also $\wt{X}$ has finite--dimensional motive.
\end{theorem}

\begin{proof} This is well--known, and follows from the blow--up formula for Chow motives \cite[Theorem 2.8]{Sc}.
\end{proof}

An essential property of varieties with finite--dimensional motive is embodied by the nilpotence theorem:

\begin{theorem}[Kimura \cite{Kim}]\label{nilp} Let $X$ be a smooth projective variety of dimension $n$ with finite--dimensional motive. Let $\Gamma\in A^n(X\times X)_{}$ be a correspondence which is numerically trivial. Then there is $N\in\NN$ such that
     \[ \Gamma^{\circ N}=0\ \ \ \ \in A^n(X\times X)_{}\ .\]
\end{theorem}

\begin{proof} This is \cite[Proposition 7.5]{Kim}.
\end{proof}

 Actually, the nilpotence property (for all powers of $X$) could serve as an alternative definition of finite--dimensional motive, as shown by 
 Jannsen \cite[Corollary 3.9]{J4}.
 
Conjecturally, any variety has finite--dimensional motive \cite[Conjecture 7.1]{Kim}; we are still far from knowing this.

\subsection{MCK decomposition}

\begin{definition}[Murre \cite{Mur}] Let $X$ be a projective quotient variety of dimension $n$. We say that $X$ has a {\em CK decomposition\/} if there exists a decomposition of the diagonal
   \[ \Delta_X= \pi_0+ \pi_1+\cdots +\pi_{2n}\ \ \ \hbox{in}\ A^n(X\times X)\ ,\]
  such that the $\pi_i$ are mutually orthogonal idempotents and $(\pi_i)_\ast H^\ast(X)= H^i(X)$.
  
  (NB: ``CK decomposition'' is shorthand for ``Chow--K\"unneth decomposition''.)
\end{definition}

\begin{remark} The existence of a CK decomposition for any smooth projective variety is part of Murre's conjectures \cite{Mur}, \cite{J2}. 
\end{remark}

\begin{definition}[Shen--Vial \cite{SV}]\label{small} Let $X$ be a projective quotient variety of dimension $n$. Let $\Delta_X^{sm}\in A^{2n}(X\times X\times X)$ be the class of the small diagonal
  \[ \Delta_X^{sm}:=\bigl\{ (x,x,x)\ \vert\ x\in X\bigr\}\ \subset\ X\times X\times X\ .\]
  An {\em MCK decomposition\/} is a CK decomposition $\{\pi^X_i\}$ of $X$ that is {\em multiplicative\/}, i.e. it satisfies
  \[ \pi^X_k\circ \Delta_X^{sm}\circ (\pi^X_i\times \pi^X_j)=0\ \ \ \hbox{in}\ A^{2n}(X\times X\times X)\ \ \ \hbox{for\ all\ }i+j\not=k\ .\]
  
 (NB: ``MCK decomposition'' is shorthand for ``multiplicative Chow--K\"unneth decomposition''.) 
  
 A {\em weak MCK decomposition\/} is a CK decomposition $\{\pi^X_i\}$ of $X$ that satisfies
    \[ \Bigl(\pi^X_k\circ \Delta_X^{sm}\circ (\pi^X_i\times \pi^X_j)\Bigr){}_\ast (a\times b)=0 \ \ \ \hbox{for\ all\ } a,b\in\ A^\ast(X)\ .\]
  \end{definition}
  
  \begin{remark} The small diagonal (seen as a correspondence from $X\times X$ to $X$) induces the {\em multiplication morphism\/}
    \[ \Delta_X^{sm}\colon\ \  h(X)\otimes h(X)\ \to\ h(X)\ \ \ \hbox{in}\ \MM_{\rm rat}\ .\]
 Suppose $X$ has a CK decomposition
  \[ h(X)=\bigoplus_{i=0}^{2n} h^i(X)\ \ \ \hbox{in}\ \MM_{\rm rat}\ .\]
  By definition, this decomposition is multiplicative if for any $i,j$ the composition
  \[ h^i(X)\otimes h^j(X)\ \to\ h(X)\otimes h(X)\ \xrightarrow{\Delta_X^{sm}}\ h(X)\ \ \ \hbox{in}\ \MM_{\rm rat}\]
  factors through $h^{i+j}(X)$.
  
  If $X$ has a weak MCK decomposition, then setting
    \[ A^i_{(j)}(X):= (\pi^X_{2i-j})_\ast A^i(X) \ ,\]
    one obtains a bigraded ring structure on the Chow ring: that is, the intersection product sends $A^i_{(j)}(X)\otimes A^{i^\prime}_{(j^\prime)}(X) $ to  $A^{i+i^\prime}_{(j+j^\prime)}(X)$.
    
      It is expected (but not proven !) that for any $X$ with a weak MCK decomposition, one has
    \[ A^i_{(j)}(X)\stackrel{??}{=}0\ \ \ \hbox{for}\ j<0\ ,\ \ \ A^i_{(0)}(X)\cap A^i_{hom}(X)\stackrel{??}{=}0\ ;\]
    this is related to Murre's conjectures B and D, that have been formulated for any CK decomposition \cite{Mur}.

  The property of having an MCK decomposition is severely restrictive, and is closely related to Beauville's ``(weak) splitting property'' \cite{Beau3}. For more ample discussion, and examples of varieties with an MCK decomposition, we refer to \cite[Section 8]{SV}, as well as \cite{V6}, \cite{SV2}, \cite{FTV}.
    \end{remark}

\subsection{The Fourier decomposition}

In what follows, we will make use of the following: 

\begin{theorem}[Shen--Vial \cite{SV}]\label{sv} Let $Y\subset\PP^5(\C)$ be a smooth cubic fourfold, and let $X:=F(Y)$ be the Fano variety of lines in $Y$. There exists a self--dual CK decomposition $\{\Pi^X_i\}$ for $X$, and 
  \[ (\Pi^X_{2i-j})_\ast A^i(X) = A^i_{(j)}(X)\ ,\]
  where the right--hand side denotes the splitting of the Chow groups defined in terms of the Fourier transform as in \cite[Theorem 2]{SV}. Moreover, we have
  \[ A^i_{(j)}(X)=0\ \ \ \hbox{if\ }j<0\ \hbox{or\ }j>i\ \hbox{or\ } j\ \hbox{is\ odd}\ .\]
  
  In case $Y$ is very general, the Fourier decomposition $A^\ast_{(\ast)}(X)$ forms a bigraded ring, and hence
  $\{\pi^X_i\}$ is a weak MCK decomposition.
    \end{theorem}

\begin{proof} (A matter of notation: what we denote $A^i_{(j)}(X)$ is denoted $CH^i(X)_j$ in \cite{SV}.)

The existence of a self--dual CK decomposition $\{\Pi^X_i\}$ is \cite[Theorem 3.3]{SV}. (More in detail: \cite[Theorem 3.3]{SV} applies to any hyperk\"ahler fourfold $F$ of $K3^{[2]}$ type with a cycle class $L\in A^2(F\times F)$ that represents the Beauville--Bogomolov pairing and satisfies \cite[equalities (6), (7), (8), (9)]{SV}. For the Fano variety of lines of a cubic fourfold, the cycle $L$ of \cite[definition (107)]{SV} has these properties, as shown in \cite[Section 3]{SV}.)

According to \cite[Theorem 3.3]{SV}, the given CK decomposition agrees with the Fourier decomposition of the Chow groups. The ``moreover'' part is because the 
$\{\Pi^X_i\}$ are shown to satisfy Murre's conjecture B \cite[Theorem 3.3]{SV}.

The statement for very general cubics is \cite[Theorem 3]{SV}.
    \end{proof}

\begin{remark}\label{pity} Unfortunately, it is not yet known that the Fourier decomposition of \cite{SV} induces a bigraded ring structure on the Chow ring for {\em all\/} Fano varieties $X$ of smooth cubic fourfolds. For one thing, it has not yet been proven that 
  \[  A^2_{(0)}(X)\cdot A^2_{(0)}(X)\ \stackrel{??}{\subset}\  A^4_{(0)}(X)\] 
  (cf. \cite[Section 22.3]{SV} for discussion).
\end{remark}

\subsection{Refined CK decomposition}

\begin{theorem}[]\label{pi20} Let $X$ be a smooth projective hyperk\"ahler fourfold of $K3^{[2]}$--type. Assume that $X$ has finite--dimensional motive. Then $X$ has a CK decomposition $\{\pi^X_i\}$. Moreover, there exists a further splitting
  \[ \pi^X_2 = \pi^X_{2,0} + \pi^X_{2,1}\ \ \ \hbox{in}\ A^4(X\times X)\ ,\]
  where $\pi^X_{2,0}$ and $\pi^X_{2,1}$ are orthogonal idempotents, and $\pi^X_{2,1}$ is supported on $C\times D\subset X\times X$, where $C$ and $D$ are a curve, resp. a divisor on $X$.
  The action on cohomology verifies
  \[ (\pi^X_{2,0})_\ast H^\ast(X) = H^2_{tr}(X)\ ,\]
  where $H^2_{tr}(X)\subset H^2(X)$ is defined as the orthogonal complement of $NS(X)$ with respect to the Beauville--Bogomolov form. The action on Chow groups verifies
  \[  (\pi^X_{2,0})_\ast A^2(X) = (\pi^X_2)_\ast A^2_{}(X)\ .\]
  \end{theorem}

  \begin{proof} It is known \cite{CM} that $X$ verifies the Lefschetz standard conjecture $B(X)$. Combined with finite--dimensionality, this implies the existence of a CK decomposition \cite[Lemma 5.4]{J2}.
  
  For the ``moreover'' statement, one observes that $X$ verifies conditions (*) and (**) of Vial's \cite{V4}, and so \cite[Theorems 1 and 2]{V4} apply. This gives the existence of refined CK projectors $\pi^X_{i,j}$, which act on cohomology as projectors on gradeds for the ``niveau filtration'' $\wt{N}^\ast$ of loc. cit. In particular, $\pi^X_{2,1}$ acts as projector on $NS(X)$, and $\pi^X_{2,0}$ acts as projector on $H^2_{tr}(X)$. The projector $\pi^X_{2,1}$, being supported on $C\times D$, acts trivially on $A^2(X)$ for dimension reasons; this proves the last equality.
   \end{proof}





     

\section{Main result}

\begin{theorem}\label{main}  Let $Y\subset\PP^5(\C)$ be a smooth cubic fourfold defined by an equation
    \[ f(X_0,X_1,X_2,X_3)+ g(X_4,X_5)=0\ ,\]
    where $f$ and $g$ are homogeneous polynomials of degree $3$. Let $X=F(Y)$ be the Fano variety of lines in $Y$. 
    Let $\sigma\in\aut(X)$ be the order $3$ non--symplectic automorphism induced by
    \[  \begin{split}\sigma_\PP\colon\ \ \  \PP^5(\C)\ &\to\ \PP^5(\C)\ ,\\
                         [X_0:\ldots:X_5]\ &\mapsto\ [X_0:X_1:X_2:X_3:\nu X_4:\nu X_5]\ ,\\
                         \end{split}\]
    where $\nu$ is a primitive $3$rd root of unity. 
    
    Then
    \[ (\ide +\sigma+\sigma^2)_\ast \ A^i_{(j)}(X)=0\ \ \ \hbox{for}\ (i,j)\in\{(2,2),(4,2),(4,4)\}\ .\]
    In particular,
     \[  (\ide +\sigma+\sigma^2)_\ast \ A^4_{hom}(X)=0  \ .\]   
     \end{theorem}

 \begin{proof}  
 (NB: the family of Fano varieties of theorem \ref{main} is described in \cite[Example 6.5]{BCS}, from which I learned that the automorphism $\sigma$ is 
 non--symplectic.)
 
 The last phrase of the theorem follows from the one--but--last phrase, since 
   \[A^4_{hom}(X)=A^4_{(2)}(X)\oplus A^4_{(4)}(X)\ \] 
   \cite[Theorem 4]{SV}.
 
 In a first step of the proof, let us show that the automorphism $\sigma$ respects (most of) the Fourier decomposition of the Chow ring:
 
 \begin{proposition}\label{compat} Let $X$ and $\sigma$ be as in theorem \ref{main}. Let $A^\ast_{(\ast)}(X)$ be the Fourier decomposition (theorem \ref{sv}).
 Then
   \[ \sigma_\ast \, A^i_{(j)}(X)\ \subset\ A^i_{(j)}(X)\ \ \ \forall (i,j)\not=(2,0)\ .\]
  \end{proposition}
  
  \begin{proof} Here, the alternative description of the Fourier decomposition $A^\ast_{(\ast)}(X)$ in terms of a certain rational map
    $ \phi\colon X \dashrightarrow X$
    comes in handy.
  
  Let $Y\subset\PP^5(\C)$ be {any\/} smooth cubic fourfold (not necessarily with automorphisms), and let $X=F(Y)$ be the Fano variety of lines in $Y$. There exists a degree $16$ rational map
    \[ \phi\colon\ \ \ X\ \dashrightarrow\ X\ \]
    \cite{V21}, \cite[Section 18]{SV}. The map $\phi$ is defined as follows: Let $x\in X$ be a point, and let $\ell\subset Y$ be the line corresponding to $x$. For a general point $x\in X$, there is a unique plane
   $H\subset\PP^5$ that is tangent to $Y$ along $\ell$. Then $\phi(x)\in X$ is defined as the point corresponding to $\ell^\prime\subset Y$, where
      \[ H\cap Y = 2\ell +\ell^\prime\ .\]

    As in \cite[Definition 21.8]{SV}, for any $\lambda\in\QQ$ let us consider the eigenspaces
      \[  V^i_\lambda:= \bigl\{  c\in A^i(X)\ \vert\ \phi^\ast(c)=\lambda \cdot c\bigr\}\ .\]
      
      These eigenspaces are related to the Fourier decomposition of the Chow ring: indeed, Shen--Vial show \cite[Theorem 21.9 and Proposition 21.10]{SV} that there is a decomposition
      \begin{equation}\label{thisdeco}  A^i_{(j)}(X) = V^i_{\lambda_1}\oplus\cdots\oplus V^i_{\lambda_r}\ \ \ \ \ \forall (i,j)\not=(2,0)\ .\end{equation}
      
     Let us now return to $X$ and $\sigma$ as in theorem \ref{main}, and let us prove proposition \ref{compat}. In view of the decomposition (\ref{thisdeco}), we see that to prove proposition \ref{compat}, it suffices to prove the following:
      
      \begin{claim}\label{phi} Let $X$ and $\sigma$ be as in theorem \ref{main}. Then
        \[ \phi^\ast \sigma^\ast = \sigma^\ast \phi^\ast\colon\ \ \ A^i(X)\ \to\ A^i(X)\ .\]
       \end{claim}
       
       In order to prove the claim, we first establish a little lemma:
       
       \begin{lemma}\label{=} Set--up as above. There is an equality of rational maps
       \[   \phi\circ\sigma =\sigma\circ\phi\colon\ \ \ X\ \dashrightarrow\ X\ .\]
       \end{lemma}
       
    \begin{proof}
    Let $x\in X$ be a point outside of the indeterminacy locus of $\phi$, and let $H\subset\PP^5$ be the plane tangent to $Y$ along the line $\ell$ corresponding to $x$. By definition, $\phi(x)\in X$ is the point corresponding to $\ell^\prime\subset Y$, where
      \[ H\cap Y = 2\ell +\ell^\prime\ .\]
      Let $\sigma_\PP\colon\PP^5\to\PP^5$ denote the linear transformation inducing the automorphism $\sigma$.
     The plane $\sigma_\PP(H)$ is tangent to $Y$ along $\sigma_\PP(\ell)$, and
      \[ \sigma_\PP(H)  \cap Y=2\sigma_\PP(\ell)+ \sigma_\PP(\ell^\prime)\ .\]
      It follows that $\phi(\sigma(x))=\sigma(\phi(x))$.
            \end{proof}
      
     Lemma \ref{=} furnishes a commutative diagram
       \[ \begin{array}[c]{ccc}
               X &\stackrel{\phi}{\dashrightarrow}& X\\ 
              \ \ \  \downarrow{\scriptstyle \sigma} &&  \ \ \  \downarrow{\scriptstyle \sigma}\\
                 X &\stackrel{\phi}{\dashrightarrow}& \ X\ .\\ 
             \end{array}\]
        This can be ``resolved'' by a commutative diagram
        \[ \begin{array}[c]{ccccc}
                X & \stackrel{p^\prime}{\leftarrow} & Z^\prime & \stackrel{q^\prime}{\to}& X\\
                 \ \ \  \downarrow{\scriptstyle \sigma} &&   \ \ \  \downarrow{\scriptstyle \sigma_Z}  &&    \ \ \  \downarrow{\scriptstyle \sigma}\\                     
                X & \stackrel{p}{\leftarrow} & Z & \stackrel{q}{\to}&\  X\ ,\\   
                \end{array}\]
              where horizontal arrows are birational morphisms such that $\phi\circ p^\prime =q^\prime$ and $\phi\circ p=q$ (and so $\phi^\ast= p_\ast q^\ast = (p^\prime)_\ast (q^\prime)^\ast\colon A^i(X)\to A^i(X)$.).
                                  
      Let us now prove claim \ref{phi}. We have equalities
      \[ \begin{split}  \phi^\ast \sigma^\ast &= (p^\prime)_\ast (q^\prime)^\ast \sigma^\ast\\
                                                                &= (p^\prime)_\ast (\sigma_Z)^\ast q^\ast\\
                                                                &= \sigma^\ast p_\ast  q^\ast\\
                                                                &= \sigma^\ast \phi^\ast\colon\ \ \ \ \ \ A^i(X)\ \to\ A^i(X)\ .\\
                                                        \end{split}\]
                           Here, in the third equality we have used the following:
                           
               \begin{sublemma} Set--up as above. There is equality
                \[ (p^\prime)_\ast (\sigma_Z)^\ast=   \sigma^\ast p_\ast\colon\ \ \ A^i( Z)\ \to\ A^i(X)  \ .\]              
               \end{sublemma}    
               
        \begin{proof} Since $\sigma_Z$ is a birational morphism, there is equality
         \[   \sigma_\ast(p^\prime)_\ast (\sigma_Z)^\ast= p_\ast (\sigma_Z)_\ast(\sigma_Z)^\ast=p_\ast \colon\ \ \ A^i(Z)\ \to\ A^i(X)\ .\]
         Composing on the left with $\sigma^\ast$, this implies
         \[  \sigma^\ast  \sigma_\ast(p^\prime)_\ast (\sigma_Z)^\ast= \sigma^\ast p_\ast\colon\ \ \ A^i(Z)\ \to\ A^i(X)\ .\]      
         But $\sigma^\ast=(\sigma^2)_\ast$ and so the left--hand side simplifies to $  (p^\prime)_\ast (\sigma_Z)^\ast$, proving the sublemma.        
        \end{proof}

  \end{proof}

  For later use, we recast proposition \ref{compat} as follows:
  
  \begin{corollary}\label{commut} Set--up as above. Let $\{\Pi_j^X\}$ be a CK decomposition as in theorem \ref{sv}. Then
   \[ \sigma_\ast (\Pi_j^X)_\ast   = (\Pi_j^X)_\ast \sigma_\ast  (\Pi_j^X)_\ast\colon\ \ \ A^i(X)\ \to\ A^i(X)\ \ \ \forall (i,j)\not=(2,4)\ .\]  
  \end{corollary}
  
  \begin{proof} This is just a translation of proposition \ref{compat}, using the fact that $\Pi_j^X$ acts on $A^i(X)$ as projector on $A^i_{(2i-j)}(X)$.
  \end{proof}

  The second step of the proof is to ascertain that $X$ has finite--dimensional motive:
  
  \begin{proposition}\label{findim} Let $Y\subset\PP^5(\C)$ and $X=F(Y)$ be as in theorem \ref{main}. Then $Y$ and $X$ have finite--dimensional motive.
  \end{proposition}
  
  \begin{proof} 
  To establish finite--dimensionality of $Y$ is an easy exercice in using what is known as the ``Shioda inductive structure'' \cite{Shi}, \cite{KS}. Indeed, applying 
 \cite[Remark 1.10]{KS}, we find there exists a dominant rational map
   \[ \phi\colon\ \ \ Y_1\times Y_2\ \dashrightarrow\ Y\ ,\]
   where $Y_1\subset\PP^3(\C)$ is the smooth cubic threefold defined as
     \[ f(X_0,X_1,X_2,X_3)+V^3=0\ ,\]
     and $Y_2\subset\PP^2(\C)$ is the smooth cubic curve defined as
               \[ g(X_0,X_1)+W^3=0\ .\]
                 The indeterminacy locus of $\phi$ is resolved by blowing up the locus $S\times P\subset Y_1\times Y_2$, where $S\subset Y_1$ is a cubic surface, and $P\subset Y_2$ is a set of points. Let us call this blow--up $\hat{Y}$. Using theorems \ref{domcurves} and \ref{blow} and an induction on the dimension, we find that
           $\hat{Y}$ has finite--dimensional motive. Since $\hat{Y}$ dominates $Y$, it follows from \cite[Proposition 6.9]{Kim} that the cubic $Y$ has finite--dimensional motive. 
           
     Finally, \cite[Theorem 4]{fano} states that for any cubic $Y$ with finite--dimensional motive, the Fano variety $X=F(Y)$ also has finite--dimensional motive.            
  \end{proof} 
 
 The third step of the proof is to show the desired statement for $A^2_{(2)}(X)$, i.e. we now prove that
   \begin{equation}\label{22}    (\ide +\sigma+\sigma^2)_\ast \ A^2_{(2)}(X)=0\ .  \end{equation}
   
   In order to do so, let us abbreviate
   \[ \Delta_G:= {1\over 3}\bigl(\Delta_X +\Gamma_\sigma +\Gamma_{\sigma\circ\sigma}\bigr)\ \ \ \in\ A^4(X\times X)\ .\]
   Since the action of $\sigma$ is non--symplectic \cite[Example 6.5 and Lemma 6.2]{BCS}, we have that
   \[ (\Delta_G)_\ast=0\colon\ \ \ H^2(X,\OO_X)\ \to\ H^2(X,\OO_X)\ .\] 
   Using the Lefschetz $(1,1)$--theorem, we see that
    \[  \Delta_G\circ \Pi_2^X =\gamma \ \ \ \hbox{in}\ H^8(X\times X)\ ,\]
    where $\gamma$ is some cycle supported on $D\times D\subset X\times X$, for some divisor $D\subset X$. In other words, the correspondence
     \[ \Gamma:=  \Delta_G\circ \Pi_2^X - \gamma  \ \ \ \in A^4(X\times X) \]
     is homologically trivial. But then (since $X$ has finite-dimensional motive) there exists $N\in\NN$ such that
     \[  \Gamma^{\circ N}=0\ \ \ \hbox{in}\ A^4(X\times X)\ .\]
     Upon developing this expression, one finds an equality    
     \[ \Gamma^{\circ N}=  (\Delta_G\circ \Pi_2^X)^{\circ N} +  \gamma^\prime=0\ \ \ \hbox{in}\ A^4(X\times X)\ ,\]
    where $\gamma^\prime$ is supported on $D\times D\subset X\times X$. In particular, $\gamma^\prime$ acts trivially on $A^2_{(2)}(X)\subset A^2_{AJ}(X)$, and so
      \[   \bigl((\Delta_G\circ \Pi_2^X)^{\circ N}\bigr){}_\ast=0\colon\ \ \ A^2_{(2)}(X)\ \to\ A^2(X)\ .\]
      Corollary \ref{commut} (combined with the fact that $\Delta_G$ and $\Pi_2^X$ are idempotents) implies that 
            \[  \bigl((\Delta_G\circ \Pi_2^X)^{\circ N}\bigr){}_\ast=   (\Delta_G\circ \Pi_2^X){}_\ast\colon\ \ \ A^i(X)\ \to\ A^i(X)\ ,\]
            and so we find that
            \[    \bigl(\Delta_G\circ \Pi_2^X\bigr){}_\ast= (\Delta_G)_\ast=  0\colon\ \ \ A^2_{(2)}(X)\ \to\ A^2(X)\ .\]     
        This proves equality (\ref{22}).    
        
   The argument for $A^4_{(2)}(X)$ is similar: the correspondence $\Gamma$ being homologically trivial, its transpose
   \[ {}^t \Gamma= \Pi_6^X\circ \Delta_G -\gamma^{\prime\prime}\ \ \ \in A^4(X\times X) \]
   is also homologically trivial (where $\gamma^{\prime\prime}$ is supported on $D\times D$). Using nilpotence and lemma \ref{commut}, this implies (just as above) that
   \begin{equation}\label{42}   ( \Pi_6^X\circ \Delta_G){}_\ast=    \bigl(\Delta_G\circ \Pi_6^X\bigr){}_\ast= (\Delta_G)_\ast=  0\colon\ \ \ A^4_{(2)}(X)\ \to\ A^4(X)\ .\end{equation}
        
  In the final step of the proof, it remains to consider the action on $A^4_{(4)}(X)$. 
  Ideally, one would like to use Vial's projector $\pi^X_{4,0}$ of \cite{V4} (mentioned in the proof of theorem \ref{pi20}). Unfortunately, this approach runs into problems (cf. remark \ref{refined}).
   We therefore proceed somewhat differently: to establish the statement for $A^4_{(4)}(X)$, we use the following proposition:
   
  \begin{proposition}\label{44} Notation as above. One has
    \[ \Delta_G\circ \Pi_4^X - R =0\ \ \ \hbox{in}\ H^8(X\times X)\ ,\]
    where $R\in A^4(X\times X)$ is a correspondence with the property that
      \[  R_\ast=0\colon\ \ \ A^4(X)\ \to\ A^4(X)\ .\]
    \end{proposition}
    
    Obviously, this proposition clinches the proof: using the nilpotence theorem, one sees that there exists $N\in\NN$ such that
    \[ \bigl(   \Delta_G\circ \Pi_4^X + R \bigr){}^{\circ N}=0\ \ \ \hbox{in}\ A^4(X\times X)\ .\]
   Developing, and applying the result to $A^4(X)$, one finds that
    \[       \bigl(  ( \Delta_G\circ \Pi_4^X)^{\circ N}\bigr){}_\ast=0\colon\ \ \ A^4(X)\ \to\ A^4(X)\ .\]
    Corollary \ref{commut} (combined with the fact that $\Delta_G$ and $\Pi_4^X$ are idempotents) implies that 
            \[  \bigl((\Delta_G\circ \Pi_4^X)^{\circ N}\bigr){}_\ast=        (\Delta_G\circ \Pi_4^X){}_\ast\colon\ \ \ A^i(X)\ \to\ A^i(X)\ \ \ \forall i\not=2\ .\]
         Therefore, we conclude that
            \[    \bigl(\Delta_G\circ \Pi_4^X\bigr){}_\ast= (\Delta_G)_\ast=  0\colon\ \ \ A^4_{(4)}(X)\ \to\ A^4(X)\ .\]   
            
  It only remains to prove proposition \ref{44}. Here, we use the fact that $X$ is of $K3^{[2]}$--type and so there is an isomorphism
    \[ H^4(X)=\hbox{Sym}^2 H^2(X)\ \]
    \cite[Proposition 3]{BD}.
   Using the truth of the standard conjectures for $X$ \cite[Theorem 1.1]{CM}, and the semi--simplicity of motives for numerical equivalence \cite[Theorem 1]{J}, this means that the map
   \[ \Delta^{sm}\colon\ \ \ h^2(X)\otimes h^2(X)\ \to\ h^4(X)\ \ \ \hbox{in}\ \MM_{\rm hom} \]
   admits a right--inverse, where $\Delta^{sm}\in A^8((X\times X)\times X)$ is as before the ``small diagonal'' (cf. definition \ref{small}).
     Let $\Psi\in A^4(X\times (X\times X))$ denote this right--inverse. 
      
      Using the splitting $\pi^X_2=\pi^X_{2,0}+\pi^X_{2,1}$ in $A^4(X\times X)$ of theorem \ref{pi20}, one obtains a splitting modulo homological equivalence of $\Pi^X_4$ in 4 components
      \[ \begin{split}  \Pi^X_4&= \Pi^X_4\circ \Delta^{sm}\circ (\pi^X_2\times\pi^X_2)\circ \Psi\circ \Pi^X_4 \\
                                          &=  \Pi^X_4\circ \Delta^{sm}\circ \bigl((\pi^X_{2,0}+\pi^X_{2,1})\times(\pi^X_{2,0}+\pi^X_{2,1})\bigr)\circ \Psi\circ \Pi^X_4 \\
                                          & = \sum_{{k,\ell\in\{0,1\}} {}}  \Pi^X_4\circ \Delta^{sm}\circ (\pi^X_{2,k}\times\pi^X_{2,\ell})\circ \Psi\circ \Pi^X_4 \\
                                          &=: \sum_{{k,\ell\in\{0,1\}} {}} \Pi^X_{4,k,\ell}\ \ \ \ \ \hbox{in}\ H^8(X\times X)\ .\\
                  \end{split}\]
        We note that (by construction) $\Pi^X_{4,0,0}$ acts as a projector on
        \[  \hbox{Sym}^2 H^2_{tr}(X)\ \ \ \subset\  \hbox{Sym}^2 H^2(X)=  H^4(X)\ .\]
        Also, we recall that $\pi^X_{2,1}$ is supported on $C\times D\subset X\times X$ (theorem \ref{pi20}), which implies that $\Pi^X_{4,k,\ell}\in A^4(X\times X)$ is supported on $X\times D$ for $(k,\ell)\not=(0,0)$.
        
        It will be convenient to consider the transpose decomposition
        \[ \Pi_4^X={}^t \Pi_4^X={}^t \Pi^X_{4,0,0}+  {}^t \Pi^X_{4,1,0}+{}^t \Pi^X_{4,0,1}+ {}^t \Pi^X_{4,1,1}\ \ \    
        \hbox{in}\ H^8(X\times X)\ \]                                              
        (where we have used that $\Pi^X_4$ is transpose--invariant, cf. theorem \ref{sv}).

    This decomposition induces in particular a decomposition
    \[ \Delta_G\circ \Pi^X_4 =  \Delta_G\circ {}^t \Pi^X_{4,0,0}+ \Delta_G\circ {}^t \Pi^X_{4,1,0}+{}^t \Delta_G\circ \Pi^X_{4,0,1}+ {}^t \Delta_G\circ \Pi^X_{4,1,1}\ \ \    
        \hbox{in}\ H^8(X\times X)\ .\]   
     The last $3$ summands in this decomposition act trivially on $A^4(X)$ (indeed, the correspondence ${}^t \Pi^X_{4,k,\ell}$ is supported on $D\times X\subset X\times X$ for $(k,\ell)\not=(0,0)$, and hence acts trivially on $A^4(X)$). These last $3$ summands will form the correspondence called $R$ in proposition \ref{44}. To prove proposition \ref{44}, it remains to establish that
     \begin{equation}\label{400}    \Delta_G\circ {}^t \Pi^X_{4,0,0} =0\ \ \ \hbox{in}\ H^8(X\times X)\ .\end{equation}
   Taking transpose, one sees this is equivalent to proving that
   \[    \Pi^X_{4,0,0}\circ \Delta_G =0\ \ \ \hbox{in}\ H^8(X\times X)\ ,\ \]
   which in turn (since applying $\sigma^\ast$ and projecting to $\hbox{Sym}^2 H^2_{tr}(X)$ commute) is equivalent to proving that
   \[ \Delta_G\circ \Pi^X_{4,0,0} =0\ \ \ \hbox{in}\ H^8(X\times X)\ .\]
   Invoking Manin's identity principle, it suffices to prove that
    \[ \sigma^\ast (c_1\cup c_2) +(\sigma^2)^\ast(c_1\cup c_2) + c_1\cup c_2=0\ \ \ \hbox{in}\ H^4(X)\ \ \ \forall c_1, c_2\in H^2_{tr}(X)\ .\]
   Thanks to the equality
    \[ c_1\cup c_2={1\over 2}\Bigl( (c_1+c_2)\cup (c_1+c_2)  - c_1\cup c_1 -c_2\cup c_2\Bigr)\ ,\] 
    it suffices to prove that
     \begin{equation}\label{c^2} \sigma^\ast (c\cup c) +(\sigma^2)^\ast(c\cup c) + c\cup c=0\ \ \ \hbox{in}\ H^4(X)\ \ \ \forall c\in H^2_{tr}(X)\ .\end{equation}
     
     We now make the following claim:
     
     \begin{claim}\label{claim} Set--up as above. Let $c\in H^2_{tr}(X)$. Then
       \[ (\sigma^\ast)(c) \cup (\sigma^2)^\ast(c) =c\cup c\ \ \ \hbox{in}\ H^4(X).\]   
       \end{claim}
       
     It is readily checked that claim \ref{claim} implies equality (\ref{c^2}) (and hence equality (\ref{400}) and hence also proposition \ref{44}): We have  
     \[ \begin{split}  \sigma^\ast (c\cup c) +(\sigma^2)^\ast(c\cup c) + c\cup c &= \bigl( \sigma^\ast(c) +(\sigma^2)^\ast(c)\bigr)^{\cup 2} -2 \sigma^\ast(c)\cup
     (\sigma^2)^\ast(c) + c\cup c\\
          &= 2 c\cup c - 2  \sigma^\ast(c)\cup
     (\sigma^2)^\ast(c)\\
          &= 0\ \ \ \ \hbox{in}\    H^4(X)\ ,\\
          \end{split}\]
          proving equality (\ref{c^2}).
          (Here, the second equality is because $\sigma^\ast(c)+(\sigma^2)^\ast(c)=-c$, and the third equality is the claim.)
          
     Let us now prove claim \ref{claim}. The point is that the subgroup
     \[  H:= \bigl\{  c\in H^2_{}(X)\ \vert\ (\sigma^\ast)(c) \cup (\sigma^2)^\ast(c) =c\cup c\  \hbox{in}\ H^4(X)\bigr\} \ \ \ \subset\ H^2_{}(X)\ , \]
together with its complexification $H_\C$, defines a sub--Hodge structure of $H^2(X)$. Let $\omega\in H^{2,0}(X)$ be a generator. Then $\omega$ is in $H_\C$ (since $\sigma^\ast\omega=\nu\cdot\omega$, with $\nu^3=1$, $\nu$ primitive). But $H^2_{tr}(X)\subset H^2(X)$ is the smallest sub--Hodge structure containing $\omega$, and so we must have 
  \[ H^2_{tr}(X)\ \subset\ H\ ,\]
  which proves claim \ref{claim}.
         \end{proof}

 \begin{remark}\label{refined} To prove the statement for $A^4_{(4)}(X)$ in the final step of the above proof,
 it would be natural to try and use Vial's projector $\pi^X_{4,0}$ of \cite[Theorems 1 and 2]{V4} (mentioned in the proof of theorem \ref{pi20}).
 However, this approach is difficult to put into practice:
  the problem is that it seems impossible to prove that 
    \[ \Delta_G\circ \pi^X_{4,0}=0\ \ \ \hbox{in}\ H^8(X\times X)\ ,\] 
    short of knowing that (1) $H^4(X)\cap F^1 =N^1 H^4(X)$, and (2)  $N^1 H^4(X)=\wt{N}^1 H^4(X)$, where $N^\ast$ is the usual coniveau filtration and $\wt{N}^\ast$ is Vial's niveau filtration. Both (1) and (2) seem difficult.
\end{remark}

\section{Some corollaries}    

\begin{corollary}\label{triv} Let $X$ and $\sigma$ be as in theorem \ref{main}. Let $Z:=X/<\sigma>$ be the quotient. Then
  \[ A^4(Z)\cong \QQ\ .\]
 \end{corollary}
 
 \begin{proof} We have a natural isomorphism $A^4(Z)\cong A^4(X)^\sigma$. But theorem \ref{main} (combined with the fact that $\sigma^\ast A^4_{(j)}(X)\subset A^4_{(j)}(X)$ for all $j$, cf. proposition \ref{compat}) implies that
   \[ A^4(X)^\sigma\ \ \ \subset\ A^4_{(0)}(X)\ .\]
   Since there exists a $\sigma$--invariant ample divisor $L\in A^1(X)$, and $L^4$ generates the $1$--dimensional $\QQ$--vector space $A^4_{(0)}(X)$, there is equality
   \[ A^4(X)^\sigma= A^4_{(0)}(X)\ .\] 
 \end{proof}

 \begin{corollary}\label{ghc} Let $X$ and $\sigma$ be as in theorem \ref{main}. Then the invariant part of cohomology
   \[ H^4(X)^\sigma  \ \subset \ H^4(X) \]
   is supported on a divisor.
   \end{corollary}
   
   \begin{proof} This follows from theorem \ref{main} by applying the Bloch--Srinivas ``decomposition of the diagonal'' argument \cite{BS}.
   For the benefit of readers not familiar with \cite{BS}, we briefly resume this argument. 
   
   Let $k\subset\C$ be a subfield such that $X$ and $\Delta_G$ are defined over $k$, and such that $k$ is finitely generated over $\QQ$. Let $k(X)$ denote the function field of $X_k$. Since there is an embedding $k(X)\subset\C$, there is a natural homomorphism
   \[ A^\ast(X_{k(X)})\ \to\ A^\ast(X_\C) \]
   that is injective \cite[Appendix to Lecture 1]{B}. In particular, there is an injective homomorphism
   \[ A^4(X_{k(X)})^\sigma\ \hookrightarrow\ A^4(X_\C)^\sigma\ .\]
   As the right--hand side has dimension $1$ (theorem \ref{main}), it follows that also
   \[ \dim A^4(X_{k(X)})^\sigma =1\ .\]
   We now consider the image of $\Delta_G\in A^4(X_k\times X_k)^{\sigma\times\sigma}$ under the restriction homomorphism
   \[ A^4(X_k\times X_k)^{\sigma\times\sigma}\ \to\ \varinjlim A^4(X_k\times U)^\sigma\cong   A^4(X_{k(X)})^\sigma =\QQ\ \]
   (here the limit is over Zariski opens $U\subset X_k$, and the isomorphism follows from \cite[Appendix to Lecture 1]{B}).    
   This gives a decomposition
   \[ \Delta_G= x\times X + \gamma\ \ \ \hbox{in}\ A^4(X_k\times X_k)\ ,\]
   where $\gamma$ is supported on $X\times D$ for some divisor $D\subset X$.
   Considering this decomposition for $X=X_\C$, and looking at the action of correspondences on cohomology, we find that
   \[ H^4(X)^\sigma = (\Delta_G)_\ast H^4(X) =  \gamma_\ast H^4(X)\ ,\]
   and thus $H^4(X)^\sigma$ is supported on the divisor $D$. 
       \end{proof}

\vskip1cm
\begin{nonumberingt} Thanks to all participants of the Strasbourg 2014/2015 ``groupe de travail'' based on the monograph \cite{Vo} for a very stimulating atmosphere. Thanks to the referee for genuinely helpful comments.
Many thanks to Kai and Len who accepted to be junior members of the Schiltigheim Math Research Institute.
\end{nonumberingt}

\end{document}